\documentclass[11pt]{amsart}

\usepackage{amsmath,amssymb,amsthm}
\usepackage{enumitem}
\usepackage{tikz}
\usepackage[margin=1.1in]{geometry}
\usepackage[colorlinks=true,linkcolor=blue,citecolor=blue,urlcolor=blue]{hyperref}

\newtheorem{theorem}{Theorem}[section]
\newtheorem{proposition}[theorem]{Proposition}
\newtheorem{lemma}[theorem]{Lemma}
\newtheorem{corollary}[theorem]{Corollary}
\theoremstyle{definition}
\newtheorem{definition}[theorem]{Definition}
\newtheorem{example}[theorem]{Example}
\newtheorem{question}[theorem]{Question}
\theoremstyle{remark}
\newtheorem{remark}[theorem]{Remark}

\makeatletter
\@namedef{subjclassname@1991}{Subject}
\@namedef{subjclassname@2000}{Subject}
\@namedef{subjclassname@2020}{Subject}
\makeatother

\numberwithin{equation}{section}

\newcommand{\R}{\mathbb{R}}
\newcommand{\Z}{\mathbb{Z}}
\newcommand{\E}{\mathbb{E}}
\renewcommand{\P}{\mathbb{P}}
\newcommand{\F}{\mathcal{F}}
\newcommand{\Hz}{\mathcal{H}_0}
\newcommand{\ip}[2]{\langle #1,#2\rangle_d}
\newcommand{\nd}[1]{\|#1\|_d}
\DeclareMathOperator{\dist}{dist}
\DeclareMathOperator{\diam}{diam}
\DeclareMathOperator{\Wid}{Wid}

\begin{document}

\title{Dirichlet problem on locally finite graphs}
\author{Mohammad Javaheri}
\address{Department of Mathematics and Statistics, Siena University\\
515 Loudon Road, Loudonville, NY 12211}
\email{mjavaheri@siena.edu}


\begin{abstract}
We study existence and uniqueness of solutions to the vertex-weighted Dirichlet problem on locally finite graphs. Let $B$ be a subset of the vertices of a connected graph $G$. The Dirichlet problem asks for a function whose discrete Laplacian is prescribed on $V\setminus B$ and whose values are prescribed on $B$. Each infinite connected component of $G\setminus B$ is called an end of $G$ relative to $B$. On finite graphs we show that the solution is the limit of an averaging process, with an explicit geometric rate given by the spectrum of the random walk, and is also the limit of the heat flow and the minimizer of an energy functional. On infinite graphs we prove two existence theorems: one with no assumption on vertex degrees, under a finite-width condition and summable data, and one for bounded data on graphs of bounded degree, where the solution is also unique among bounded functions. If there are no ends, the Dirichlet problem has exactly one solution. If there are $l$ ends, the solution set, when nonempty, has dimension at least $l$; each end carries a positive harmonic function vanishing on $B$. Calling an end parabolic when the random walk started in it hits $B$ almost surely, we show that this function can be taken bounded exactly on non-parabolic ends and is necessarily unbounded on parabolic ones, a discrete analogue of a theorem of Li and Tam. Consequently, bounded solutions are unique if and only if every end is parabolic.
\end{abstract}

\maketitle

\section{Introduction}

\subsection*{Classical background}
Let $U\subset\R^n$ be a bounded region with boundary $\partial U$, and let $F\colon U\to\R$ and $h\colon\partial U\to\R$ be given. The classical Dirichlet problem asks for $\Phi\in C^2(U)\cap C^0(\bar U)$ with
\begin{equation}\label{eq:classical}
\begin{cases}
\Delta\Phi(x)=F(x), & x\in U,\\
\Phi(x)=h(x), & x\in\partial U.
\end{cases}
\end{equation}
Under mild conditions on $\partial U$, $F$ and $h$ (for example, if all are smooth) there is a unique solution \cite[Th.~6.13]{GT}. A similar result holds for bounded regions in Riemannian manifolds. On a compact manifold without boundary, the equation $\Delta\Phi=F$ has either no solution or infinitely many.

There is no complete picture on unbounded regions. Nonconstant bounded harmonic functions exist on certain noncompact manifolds \cite{Choi}. If $E$ is an end of a complete manifold, that is, a noncompact component of the complement of a compact set, with compact boundary $\partial E$, then Li and Tam \cite{LiTam} distinguish \emph{parabolic} and \emph{non-parabolic} ends. A non-parabolic end carries a bounded nonconstant harmonic function with prescribed boundary behaviour; on a parabolic end one constructs a nonnegative harmonic function vanishing on $\partial E$ and tending to infinity along a divergent sequence. One of the aims of this paper is to prove the discrete analogue of this dichotomy (Theorems~\ref{thm:ends} and~\ref{thm:parabolic}).

\subsection*{The discrete problem}
Let $G$ be a connected graph with at least two vertices, without loops or multiple edges, and write $V=V(G)$. We write $v\sim w$ if $v$ and $w$ are joined by an edge, $d_v$ for the degree of $v$, and $\F$ for the space of real functions on $V$. Given $B\subseteq V$, $f\in\F$ and $\varphi_B\colon B\to\R$, the \emph{discrete Dirichlet problem} is to find $\varphi\in\F$ with
\begin{equation}\label{eq:dirichlet}
\begin{cases}
\Delta\varphi(v)=f(v), & v\in V\setminus B,\\
\varphi(v)=\varphi_B(v), & v\in B,
\end{cases}
\end{equation}
where $\Delta$ is the vertex-weighted Laplacian of Section~\ref{sec:setup}. We call $B$ the boundary and $\varphi_B$ the boundary condition. When $f=0$, solutions are called harmonic (on $V\setminus B$).

Harmonic functions on graphs are intimately tied to random walks \cite{DoyleSnell,Lovasz,Woess}. For example, if $X_t$ is simple random walk started at $v$ and $\tau$ is its first visit to $B$, then $\varphi(v)=\E_v[\varphi_B(X_\tau)]$ is harmonic whenever $\tau<\infty$ almost surely. This interpretation runs through the whole paper: the operator that drives our averaging process is the transition matrix of the walk killed on $B$ (Section~\ref{sec:setup}), and on infinite graphs the uniqueness question is governed by whether the walk hits $B$.

\subsection*{Summary of results}
Throughout, after the change of variables of Lemma~\ref{lem:weights}, all weights are equal to one.

\emph{Finite graphs} (Sections~\ref{sec:finite} and~\ref{sec:heat}).
\begin{enumerate}[label=(\roman*)]
\item If $B\neq\emptyset$ the Dirichlet problem has exactly one solution; if $B=\emptyset$ it is solvable if and only if $\sum_v f(v)=0$, and then the solution is unique up to an additive constant (Proposition~\ref{prop:finite}).
\item The iterates of the averaging map $g$ of \eqref{eq:g} converge to the solution at a geometric rate determined by the spectrum of the random walk, except when $B=\emptyset$ and $G$ is bipartite, where the even and odd iterates converge separately to two explicitly identified limits. The lazy averaging map always converges (Theorem~\ref{thm:rate}, Corollary~\ref{cor:lazy}).
\item The solution is the long-time limit of the heat equation with exponential rate given by the first Dirichlet eigenvalue, and the minimizer of an energy functional (Propositions~\ref{prop:heat} and~\ref{prop:energy}).
\end{enumerate}

\emph{Infinite, locally finite graphs} (Sections~\ref{sec:existence} and~\ref{sec:ends}). Let $W=\Wid(G,B)=\sup_v\dist(v,B)$.
\begin{enumerate}[label=(\roman*),resume]
\item If $W<\infty$ and the data satisfy $\sum_{v\notin B}|f(v)|<\infty$ and $\sum_{v\in B}d_v\varphi_B(v)^2<\infty$, there is a solution bounded by $24\Lambda W^2$, with no assumption on degrees (Theorem~\ref{thm:existence}). For $f=0$ and bounded $\varphi_B$ a bounded harmonic extension exists for every nonempty $B$ (Corollary~\ref{cor:harmonic}).
\item If degrees are bounded by $D$, $W<\infty$, and $f,\varphi_B$ are bounded, there is exactly one bounded solution, and it satisfies $\sup|\varphi|\le\sup|\varphi_B|+WD^W\sup|f|$ (Theorem~\ref{thm:rw}).
\item If $G$ has no ends relative to $B\neq\emptyset$, the Dirichlet problem has exactly one solution (Theorem~\ref{thm:noends}).
\item If $G$ has $l\in[0,\infty]$ ends, the space $\Hz$ of solutions of the homogeneous problem has dimension at least $l$; each end carries an element of $\Hz$ that is positive on the end and zero elsewhere (Theorem~\ref{thm:ends}).
\item This element can be taken bounded if and only if the end is non-parabolic; on a parabolic end every such element is unbounded. Bounded solutions of \eqref{eq:dirichlet} are unique if and only if every end is parabolic (Theorem~\ref{thm:parabolic}).
\end{enumerate}
Examples in Section~\ref{sec:ends} show that the lower bound $l$ in (vii) is attained in some cases and far from sharp in others.

\subsection*{Relation to the literature}
The finite-graph results of Sections~\ref{sec:finite}--\ref{sec:heat} are close to classical material on electrical networks and random walks \cite{DoyleSnell,Lovasz,ChungSGT,Grigoryan}. Explicit solutions of Dirichlet and Poisson problems on finite networks via equilibrium measures and Green functions are given by Bendito, Carmona and Encinas \cite{BCE}, and discrete Green functions with boundary are studied by Chung and Yau \cite{ChungYau}; the operator $S$ in \eqref{eq:neumann} is such a Green function. We include complete proofs because they are short and because the quantitative forms are used later.

On infinite graphs, potential theory on networks \cite{Soardi,Woess,LyonsPeres,KLW} studies the Dirichlet problem, the boundary at infinity, and recurrence. Holopainen and Soardi \cite{HolopainenSoardi} relate ends, parabolicity and $p$-harmonic functions on graphs and manifolds. Perkins \cite{Perkins} solves a Dirichlet problem on graphs with finitely many ends in which boundary values are also prescribed at the ends; in the non-parabolic case this restores the uniqueness that, by Theorem~\ref{thm:parabolic}, is lost otherwise. Energy minimization on infinite graphs is treated in \cite{Soardi,KLW} and, recently, in \cite{Tang}. The contributions of the present paper on infinite graphs are the degree-free existence theorem (Theorem~\ref{thm:existence}), the elementary construction of positive harmonic functions on each end (Theorem~\ref{thm:ends}), and the formulation of the uniqueness question in terms of parabolic ends (Theorem~\ref{thm:parabolic}), which makes precise the analogy with \cite{LiTam}.

\begin{remark}[On this version]
This version replaces an earlier one. It corrects two errors: the earlier Lemma~2.4 was false as stated (see Remark~\ref{rem:oldlemma}), and so was the second assertion of the earlier maximum principle (see Remark~\ref{rem:oldmax}). It repairs a gap in the proof of the end-counting result, proves the uniqueness statement for graphs without ends that was announced but not proved, and adds convergence rates, the random-walk existence theorem and the parabolicity results.
\end{remark}

\section{Setup}\label{sec:setup}

\subsection{The vertex-weighted Laplacian}
Let $\alpha_v>0$ be a weight at each vertex and set $\mu_v=\sqrt{\alpha_v}$. Following the vertex-weighted framework of Chung and Langlands \cite{ChungLanglands}, we consider the matrix
\begin{equation}\label{eq:L}
L(u,v)=
\begin{cases}
d_v\alpha_v, & u=v,\\
-\sqrt{\alpha_u\alpha_v}, & u\sim v,\\
0, & \text{otherwise},
\end{cases}
\end{equation}
and the corresponding operator on $\F$:
\begin{equation}\label{eq:Delta}
\Delta\varphi(v)=\sum_{w\in V}L(w,v)\varphi(w)=\mu_v\sum_{w\sim v}\bigl(\mu_v\varphi(v)-\mu_w\varphi(w)\bigr).
\end{equation}
The standard combinatorial Laplacian is the case $\mu\equiv1$. With this sign convention $\Delta$ is positive semidefinite on finitely supported functions; it is the negative of the analyst's Laplacian, which is why the heat equation in Section~\ref{sec:heat} reads $\partial_tX=-\Delta X$.

\begin{lemma}\label{lem:weights}
Let $\psi=\mu\varphi$, $\psi_B=\mu\varphi_B$ and $f'=f/\mu$. Then $\varphi$ solves \eqref{eq:dirichlet} for the weighted Laplacian with data $(f,\varphi_B)$ if and only if $\psi$ solves \eqref{eq:dirichlet} for the standard Laplacian with data $(f',\psi_B)$. The correspondence is a linear bijection between solution sets and preserves the dimension of the space of solutions of the homogeneous problem.
\end{lemma}

\begin{proof}
By \eqref{eq:Delta}, $\Delta\varphi(v)=\mu_v\sum_{w\sim v}(\psi(v)-\psi(w))$, so $\Delta\varphi=f$ on $V\setminus B$ if and only if $\Delta_{\rm std}\psi=f/\mu$ there. Multiplication by $\mu$ is a linear bijection of $\F$.
\end{proof}

Every result below is stated for $\mu\equiv1$ and transfers to the weighted case through Lemma~\ref{lem:weights}. In particular, the solvability condition $\sum_vf(v)=0$ of Proposition~\ref{prop:finite} reads $\sum_vf(v)/\mu_v=0$ in the original variables. Ends and their parabolicity (Section~\ref{sec:ends}) are notions about the graph and the normalized problem, so they are unaffected.

\subsection{Operator form and the random walk}
From now on $\mu\equiv1$ and $\Delta\varphi(v)=\sum_{w\sim v}(\varphi(v)-\varphi(w))$. Define $T\in\F$ and a linear map $A\colon\F\to\F$ by
\begin{equation}\label{eq:TA}
T(v)=\begin{cases}d_v^{-1}f(v), & v\notin B,\\ \varphi_B(v), & v\in B,\end{cases}
\qquad
(A\varphi)(v)=\begin{cases}d_v^{-1}\sum_{w\sim v}\varphi(w), & v\notin B,\\ 0, & v\in B.\end{cases}
\end{equation}
Then $\varphi$ solves \eqref{eq:dirichlet} if and only if
\begin{equation}\label{eq:IA}
(I-A)\varphi=T,
\end{equation}
that is, if and only if $\varphi$ is a fixed point of
\begin{equation}\label{eq:g}
g(\varphi)=T+A\varphi.
\end{equation}
When $f=0$ and $B=\emptyset$, $g$ replaces the value at each vertex by the average of the values at its neighbours.

Let $P(v,w)=d_v^{-1}$ for $w\sim v$ and $0$ otherwise be the transition matrix of simple random walk $(X_t)_{t\ge0}$ on $G$. We write $\P_v$ and $\E_v$ for probability and expectation when $X_0=v$, and
\[
\tau=\inf\{t\ge0: X_t\in B\}\in[0,\infty]
\]
for the hitting time of $B$. Then $A$ is the matrix $P$ with the rows indexed by $B$ replaced by zero: the walk killed on $B$. Formally, iterating \eqref{eq:IA} gives the Neumann series
\begin{equation}\label{eq:neumann}
S=\sum_{i=0}^\infty A^i,\qquad \varphi=ST,
\end{equation}
and, by first-step analysis, $ST$ is the function
\begin{equation}\label{eq:rwformula}
\varphi(v)=\E_v\bigl[\varphi_B(X_\tau);\,\tau<\infty\bigr]+\E_v\Bigl[\sum_{0\le t<\tau}\frac{f(X_t)}{d_{X_t}}\Bigr],
\end{equation}
whenever the right-hand side converges absolutely. On finite graphs with $B\neq\emptyset$ this is always the case (Section~\ref{sec:finite}); on infinite graphs it is the basis of Theorem~\ref{thm:rw}.

We use the inner product and norm
\[
\ip{\varphi}{\psi}=\sum_{v\in V}d_v\varphi(v)\psi(v),\qquad \nd{\varphi}=\ip{\varphi}{\varphi}^{1/2},
\]
on finite graphs. Since $\ip{P\varphi}{\psi}=\sum_v\sum_{w\sim v}\varphi(w)\psi(v)$ is symmetric in $\varphi,\psi$, the operator $P$ is self-adjoint for $\ip\cdot\cdot$; this is reversibility of the walk.

\section{Finite graphs: existence, uniqueness and approximation}\label{sec:finite}

Throughout this section $G$ is finite and connected, $|V|=n\ge2$, and $B\subseteq V$ is arbitrary (possibly empty).

\begin{lemma}\label{lem:eigen}
If $\lambda\in\mathbb C$ is an eigenvalue of $A$, then $|\lambda|\le1$, and $|\lambda|=1$ is possible only if $B=\emptyset$.
\end{lemma}

\begin{proof}
Let $\varphi\neq0$ be a complex function with $A\varphi=\lambda\varphi$, and choose $v$ with $|\varphi(v)|=\max_u|\varphi(u)|$. If $\lambda=0$ there is nothing to prove. Otherwise $\varphi=\lambda^{-1}A\varphi$ vanishes on $B$, so $v\notin B$ (as $\varphi\ne0$), and
\[
|\lambda|\,|\varphi(v)|=|(A\varphi)(v)|\le\frac1{d_v}\sum_{w\sim v}|\varphi(w)|\le|\varphi(v)|,
\]
so $|\lambda|\le1$. Suppose $|\lambda|=1$ and $B\ne\emptyset$. Then equality holds throughout, so $|\varphi(w)|=|\varphi(v)|$ for all $w\sim v$, and each such $w$ is again a maximum point. Repeating along a path from $v$ to $B$ we find $w\in B$ with $|\varphi(w)|=|\varphi(v)|>0$, contradicting $\varphi|_B=0$.
\end{proof}

If $B=\emptyset$, the constant function $1$ is an eigenvector of $A=P$ with eigenvalue $1$. We call $f$ \emph{closed} if $\sum_{v\in V}f(v)=0$, and set
\[
\F_0=\Bigl\{\varphi\in\F:\sum_{v\in V}d_v\varphi(v)=0\Bigr\},
\]
the $\ip\cdot\cdot$-orthogonal complement of the constants.

\begin{proposition}\label{prop:finite}
Let $G$ be a finite connected graph with boundary $B$.
\begin{enumerate}[label=(\roman*)]
\item If $B\neq\emptyset$, \eqref{eq:dirichlet} has a unique solution.
\item If $B=\emptyset$ and $f$ is closed, then for each $c\in\R$ there is a unique solution $\varphi$ with $\sum_v d_v\varphi(v)=c$.
\item If $B=\emptyset$ and $f$ is not closed, there is no solution.
\end{enumerate}
\end{proposition}

\begin{proof}
(i) By Lemma~\ref{lem:eigen}, $1$ is not an eigenvalue of $A$, so $I-A$ is invertible; use \eqref{eq:IA}.

(iii) If $B=\emptyset$ and $\varphi$ is a solution, then $d_v\varphi(v)=f(v)+\sum_{w\sim v}\varphi(w)$, and summing over $v$ gives $\sum_vd_v\varphi(v)=\sum_vf(v)+\sum_wd_w\varphi(w)$, so $f$ is closed.

(ii) Since $\sum_vd_v(A\varphi)(v)=\sum_v\sum_{w\sim v}\varphi(w)=\sum_wd_w\varphi(w)$, the map $A$ preserves $\F_0$. If $\varphi\in\F_0$ and $A\varphi=\varphi$, the maximum argument of Lemma~\ref{lem:eigen} shows that $\varphi$ is constant, hence $\varphi=0$. So $I-A$ is invertible on $\F_0$. Since $f$ is closed, $T=f/d\in\F_0$, and there is $\varphi_0\in\F_0$ with $(I-A)\varphi_0=T$. Then $\varphi=\varphi_0+c/\sum_vd_v$ is a solution with $\sum_vd_v\varphi(v)=c$. Two such solutions differ by an element of $\F_0$ fixed by $A$, which is $0$.
\end{proof}

\subsection{Convergence of the averaging process, with rates}

By Lemma~\ref{lem:eigen}, when $B\ne\emptyset$ the series \eqref{eq:neumann} converges and $g^{(n)}(\varphi)=\bigl(\sum_{i=0}^{n-1}A^i\bigr)T+A^n\varphi$ converges to the solution for every starting point $\varphi$. The following theorem makes this quantitative and treats $B=\emptyset$.

Since $P$ is self-adjoint for $\ip\cdot\cdot$, its eigenvalues are real:
\[
1=\lambda_1>\lambda_2\ge\dots\ge\lambda_n\ge-1.
\]
The eigenvalue $1$ is simple because $G$ is connected, and $\lambda_n=-1$ if and only if $G$ is bipartite \cite{Lovasz,ChungSGT}. Let $I=V\setminus B$ and, for $B\ne\emptyset$, let $A_{II}$ be the restriction of $A$ to functions supported on $I$, i.e.\ the matrix $(P(v,w))_{v,w\in I}$. It is self-adjoint for $\ip\cdot\cdot$ restricted to $I$. Let $\rho_B$ be its spectral radius. If $A_{II}\psi=\lambda\psi$ with $\lambda\ne0$, extending $\psi$ by zero gives an eigenvector of $A$ with the same eigenvalue; hence $\rho_B<1$ by Lemma~\ref{lem:eigen}.

\begin{theorem}\label{thm:rate}
Let $G$ be a finite connected graph, $\varphi\in\F$, and $g$ as in \eqref{eq:g}.
\begin{enumerate}[label=(\roman*)]
\item If $B\neq\emptyset$ and $y$ is the solution of \eqref{eq:dirichlet}, then for $n\ge1$
\[
\nd{g^{(n)}(\varphi)-y}\le\rho_B^{\,n-1}\,\nd{g(\varphi)-y}.
\]
\item If $B=\emptyset$, $f$ is closed and $G$ is not bipartite, let $y$ be the solution with $\varphi-y\in\F_0$. Then, with $\rho=\max(|\lambda_2|,|\lambda_n|)<1$,
\[
\nd{g^{(n)}(\varphi)-y}\le\rho^{\,n}\,\nd{\varphi-y}.
\]
\item If $B=\emptyset$, $f$ is closed and $G$ is bipartite with parts $P_0,Q_0$, let $y$ be as in (ii), let $\chi=1$ on $P_0$ and $\chi=-1$ on $Q_0$, and set
\[
c=\frac{\ip{\varphi-y}{\chi}}{\ip{\chi}{\chi}},\qquad z=y-c\chi.
\]
Then $g^{(2)}(z)=z$, $g(z)=y+c\chi$, $z\in\varphi+\F_0$, and with $\rho'=\max\{|\lambda|:\lambda\ne\pm1\}$ (and $\rho'=0$ if there is no such eigenvalue),
\[
\nd{g^{(2m+1)}(\varphi)-z}\le\rho'^{\,2m+1}\nd{\varphi-y},\qquad \nd{g^{(2m)}(\varphi)-g(z)}\le\rho'^{\,2m}\nd{\varphi-y}.
\]
In particular the odd iterates converge to $z$ and the even iterates to $g(z)$; the sequence converges if and only if $c=0$.
\end{enumerate}
\end{theorem}

\begin{proof}
In all cases $g(\varphi)-g(\psi)=A(\varphi-\psi)$, so $g^{(n)}(\varphi)-y=A^n(\varphi-y)$ whenever $g(y)=y$.

(i) Since $g(\varphi)|_B=\varphi_B=y|_B$, the function $e=g(\varphi)-y$ vanishes on $B$. The map $A$ preserves the space of functions vanishing on $B$ and acts there as $A_{II}$. Hence $g^{(n)}(\varphi)-y=A_{II}^{\,n-1}e$, and $\|A_{II}^{\,n-1}\|\le\rho_B^{\,n-1}$ because $A_{II}$ is self-adjoint.

(ii) Here $A=P$, and $s=\varphi-y\in\F_0$ is orthogonal to the eigenvector $1$. Expanding $s$ in a $\ip\cdot\cdot$-orthonormal eigenbasis gives $\nd{P^ns}\le\rho^n\nd s$.

(iii) Since $G$ is bipartite, $(P\chi)(v)=d_v^{-1}\sum_{w\sim v}\chi(w)=-\chi(v)$. The eigenvalue $-1$ is simple: if $P\psi=-\psi$ then $P(\chi\psi)=\chi\psi$, so $\chi\psi$ is constant. Moreover $\ip{\chi}{1}=\sum_{P_0}d_v-\sum_{Q_0}d_v=|E|-|E|=0$, so $\chi\in\F_0$ and $z\in\varphi+\F_0$. Write $s=\varphi-y=c\chi+r$ with $r\perp\chi$ and $r\perp1$. Then $P^ns=(-1)^nc\chi+P^nr$ and $\nd{P^nr}\le\rho'^{\,n}\nd r\le\rho'^{\,n}\nd s$. Thus $g^{(n)}(\varphi)=y+(-1)^nc\chi+O(\rho'^{\,n})$, which gives the two estimates since $z=y-c\chi$ and $g(z)=T+Ay-cA\chi=y+c\chi$. Finally $g^{(2)}(z)=g(y+c\chi)=y-c\chi=z$.
\end{proof}

Part (iii) recovers, with explicit limits and rate, the alternating behaviour illustrated by the function equal to $a$ on $P_0$ and $b$ on $Q_0$ with $f=0$: there $y\equiv(a+b)/2$ and $c=(a-b)/2$. Quantitative bounds on $1-\lambda_2$ and on $1-\rho_B$ follow from Cheeger-type inequalities \cite{ChungSGT}, which turn Theorem~\ref{thm:rate} into explicit estimates on how many averaging steps are needed for a given accuracy.

The bipartite obstruction disappears for the lazy averaging map.

\begin{corollary}\label{cor:lazy}
Let $B=\emptyset$, $f$ closed, and $g_{1/2}(\varphi)=\tfrac12\bigl(\varphi+g(\varphi)\bigr)$. Then $g_{1/2}$ has the same fixed points as $g$, and for every finite connected graph $G$, bipartite or not,
\[
\nd{g_{1/2}^{(n)}(\varphi)-y}\le\Bigl(\frac{1+\lambda_2}{2}\Bigr)^n\nd{\varphi-y},
\]
where $y$ is the solution with $\varphi-y\in\F_0$.
\end{corollary}

\begin{proof}
Fixed points coincide since $g_{1/2}(\varphi)=\varphi$ if and only if $g(\varphi)=\varphi$. The linear part $\frac12(I+P)$ is self-adjoint with eigenvalues $\frac12(1+\lambda_i)\in[0,1]$, and on $\F_0$ its spectral radius is $\frac12(1+\lambda_2)<1$. Argue as in Theorem~\ref{thm:rate}(ii).
\end{proof}

\begin{remark}\label{rem:oldlemma}
The earlier version of this paper proved Theorem~\ref{thm:rate}(ii) by a compactness argument based on the claim that, for non-bipartite $G$, $\|As\|_1=\|s\|_1$ with $s\in\F_0$ forces $s=0$, where $\|s\|_1=\sum_vd_v|s(v)|$. This claim is false. Let $G$ be the triangle $z_1z_2z_3$ with two pendant vertices $x\sim z_1$ and $y\sim z_2$, and let $s(x)=1$, $s(y)=-1$, $s=0$ elsewhere. Then $s\in\F_0$, $As=(0,\frac13,-\frac13,0,0)$ in the order $(x,z_1,z_2,z_3,y)$, and $\|As\|_1=2=\|s\|_1$. The flaw is that equality only forces the \emph{nonzero} values of $s$ at the neighbours of each vertex to share a sign. The correct statement is: \emph{if $G$ is non-bipartite with diameter $\delta$, then $\|A^ms\|_1<\|s\|_1$ for every $m\ge4\delta+1$ and every nonzero $s\in\F_0$.} Indeed, a shortest odd cycle has length at most $2\delta+1$, so any two vertices are joined by walks of both parities of length at most $4\delta+1$, hence (padding with back-and-forth steps) by walks of every length $m\ge4\delta+1$. Thus $K(v,w)=d_vA^m(v,w)$ is a positive symmetric matrix with column sums $d_w$, and
\[
\|A^ms\|_1=\sum_v\Bigl|\sum_wK(v,w)s(w)\Bigr|<\sum_w d_w|s(w)|=\|s\|_1,
\]
the inequality being strict because $s$ takes both signs. With $g^{(m)}$ in place of $g$, the original compactness argument then goes through. The spectral proof above is shorter and also gives the rate.
\end{remark}

\subsection{Maximum principle}

For $S\subseteq V$ let $G[S]$ be the induced subgraph, and for a connected component $C$ of $G[V\setminus B]$ let $\partial C=\{b\in B: b\sim c\text{ for some }c\in C\}$.

\begin{proposition}[Maximum principle]\label{prop:max}
Let $G$ be finite and connected, $B\neq\emptyset$, and let $\varphi$ satisfy $\Delta\varphi\le0$ on $V\setminus B$. Then $\max_V\varphi=\max_B\varphi$. If the maximum is attained at $v\notin B$, then $\varphi$ is constant on $C\cup\partial C$, where $C$ is the component of $G[V\setminus B]$ containing $v$. The analogous statements hold for minima of functions with $\Delta\varphi\ge0$.
\end{proposition}

\begin{proof}
Let $M=\max_V\varphi$ be attained at $v\notin B$. Since $\Delta\varphi(v)\le0$, $\varphi(v)\le d_v^{-1}\sum_{w\sim v}\varphi(w)\le M$, so $\varphi(w)=M$ for every $w\sim v$. Propagating through the connected set $C$, $\varphi\equiv M$ on $C$, and every vertex of $\partial C$, being a neighbour of a vertex of $C$, also has value $M$. Since $G$ is connected and $B\ne\emptyset$, $\partial C\ne\emptyset$, so $M$ is attained on $B$.
\end{proof}

\begin{remark}\label{rem:oldmax}
The earlier version asserted that $\varphi$ is constant on $V\setminus B^\circ$, where $B^\circ$ is the set of boundary vertices without interior neighbours. This fails when $G[V\setminus B]$ is disconnected: on the path $b_1-x-b_2-y-b_3$ with $B=\{b_1,b_2,b_3\}$ and boundary values $1,1,0$, the harmonic extension has $\varphi(x)=1=\max\varphi$ and $\varphi(y)=\frac12$, while $B^\circ=\emptyset$.
\end{remark}

\section{Heat equation and energy minimization}\label{sec:heat}

Throughout this section $G$ is finite and connected.

\subsection{Heat equation}
A solution of the heat equation with data $(X_0,f)$ is a function $X\colon[0,\infty)\times V\to\R$ with $X(0,\cdot)=X_0$ and
\begin{equation}\label{eq:heat}
\partial_tX(t,v)=\begin{cases}-\Delta X(t,\cdot)(v)+f(v), & v\notin B,\\ 0, & v\in B.\end{cases}
\end{equation}
This is a linear system of ordinary differential equations with a unique global solution. It is the continuous-time counterpart of the averaging process: $e^{-t\Delta}$ is the semigroup of the continuous-time random walk that jumps from $v$ at rate $d_v$, which has no periodicity. This is why no bipartite obstruction appears below.

Let $L_{II}$ be the restriction of $\Delta$ to functions supported on $I=V\setminus B$, i.e.\ the matrix $(L(v,w))_{v,w\in I}$, and $|\cdot|_2$ the Euclidean norm.

\begin{proposition}\label{prop:heat}
Let $X$ solve \eqref{eq:heat}.
\begin{enumerate}[label=(\roman*)]
\item If $B\ne\emptyset$, let $z$ be the solution of \eqref{eq:dirichlet} with $\varphi_B=X_0|_B$. Then $L_{II}$ is symmetric positive definite, and with $\lambda_1^D>0$ its smallest eigenvalue,
\[
|X(t)-z|_2\le e^{-\lambda_1^Dt}|X_0-z|_2.
\]
\item If $B=\emptyset$ and $f$ is closed, then $\sum_vX(t,v)$ is constant in $t$, and $X(t)\to z_*$, the unique solution of $\Delta z_*=f$ with $\sum_vz_*(v)=\sum_vX_0(v)$, at rate $e^{-a(G)t}$, where $a(G)$ is the smallest nonzero eigenvalue of $\Delta$.
\item If $B=\emptyset$ and $f$ is not closed, let $\bar f=|V|^{-1}\sum_vf(v)$. Then $\sum_vX(t,v)=\sum_vX_0(v)+t|V|\bar f$ is unbounded, and $X(t)-t\bar f$ converges.
\end{enumerate}
\end{proposition}

\begin{proof}
(i) Since $\partial_tX=0$ on $B$, $X(t)|_B=X_0|_B=z|_B$. Set $Y=X-z$; then $Y|_B=0$ and, on $I$, $\partial_tY=-\Delta X+f=-\Delta Y=-L_{II}Y$, because $Y$ vanishes on $B$. The quadratic form of $L_{II}$ is
\[
\sum_{\substack{v\sim w\\ v,w\in I}}(Y(v)-Y(w))^2+\sum_{v\in I}|\{b\in B:b\sim v\}|\,Y(v)^2 ,
\]
which vanishes only if $Y$ is constant on each component of $G[I]$ and zero at vertices adjacent to $B$; every component has such a vertex since $G$ is connected and $B\ne\emptyset$. So $L_{II}$ is positive definite and $Y(t)=e^{-tL_{II}}Y(0)$ decays at rate $\lambda_1^D$.

(ii) $\partial_t\sum_vX=-\sum_v\Delta X(v)+\sum_vf(v)=0$, since $\sum_v\Delta\varphi(v)=0$ for every $\varphi$. Let $z$ be any solution of $\Delta z=f$ (Proposition~\ref{prop:finite}) and $Y=X-z$, so $\partial_tY=-\Delta Y$. As $\Delta$ is symmetric positive semidefinite with kernel the constants, $Y(t)$ converges at rate $e^{-a(G)t}$ to the constant $|V|^{-1}\sum_vY(0,v)$. Hence $X(t)\to z_*$, characterized by conservation of $\sum_vX$.

(iii) $U=X-t\bar f$ solves \eqref{eq:heat} with $f$ replaced by the closed function $f-\bar f$, since $\Delta$ kills constants; apply (ii).
\end{proof}

\subsection{Energy minimization}
Let $S_B=\{\varphi\in\F:\varphi|_B=\varphi_B\}$ (so $S_\emptyset=\F$), let
\[
\mathcal E(\varphi)=\sum_{\{v,w\}:\,v\sim w}(\varphi(v)-\varphi(w))^2=\tfrac12\sum_{v\in V}\sum_{w\sim v}(\varphi(v)-\varphi(w))^2,
\]
and $H(\varphi)=\mathcal E(\varphi)-2\sum_{v\notin B}f(v)\varphi(v)$. Since $\partial\mathcal E/\partial\varphi(v)=2\Delta\varphi(v)$, the critical points of $H$ on $S_B$ are exactly the solutions of \eqref{eq:dirichlet}; as $H$ is convex, they are exactly its minimizers on $S_B$. Let $\delta=\diam G$.

\begin{proposition}\label{prop:energy}
$H$ attains its minimum on $S_B$ if $B\ne\emptyset$, or if $B=\emptyset$ and $f$ is closed. If $B=\emptyset$ and $f$ is not closed, $H$ is unbounded below on $\F$. This gives a second proof of the existence statements in Proposition~\ref{prop:finite}.
\end{proposition}

\begin{proof}
For any $u,w\in V$, joining them by a path of length $\ell\le\delta$ and using Cauchy--Schwarz,
\begin{equation}\label{eq:cs}
(\varphi(u)-\varphi(w))^2\le\delta\,\mathcal E(\varphi).
\end{equation}

Case $B=\emptyset$, $f$ closed. Since $H(\varphi+c)=H(\varphi)$, we may assume $\min\varphi=0$; let $M=\max\varphi$. By \eqref{eq:cs}, $M^2\le\delta\mathcal E(\varphi)$. As $f$ is closed, $|2\sum_vf(v)\varphi(v)|=|2\sum_vf(v)(\varphi(v)-M/2)|\le M\sum_v|f(v)|$. Hence
\[
H(\varphi)\ge\tfrac12\mathcal E(\varphi)+\tfrac{1}{2\delta}M^2-M\sum_v|f(v)|\ge\tfrac12\mathcal E(\varphi)-\tfrac{\delta}{2}\Bigl(\sum_v|f(v)|\Bigr)^2.
\]
Along a minimizing sequence normalized by $\min\varphi_i=0$, $\mathcal E(\varphi_i)$ is bounded, hence so is $\max\varphi_i$ by \eqref{eq:cs}; a convergent subsequence has a minimizer as its limit.

Case $B\ne\emptyset$. For $\varphi\in S_B$ and $v\in V$, \eqref{eq:cs} with $w\in B$ gives $|\varphi(v)|\le\beta+\sqrt{\delta\mathcal E(\varphi)}$, where $\beta=\max_B|\varphi_B|$. Hence $H(\varphi)\ge\mathcal E(\varphi)-2\sum|f|\,(\beta+\sqrt{\delta\mathcal E(\varphi)})$, which is bounded below and forces $\mathcal E$, hence $\varphi$, to stay bounded along a minimizing sequence; conclude as before.

Finally, if $B=\emptyset$ and $\sum f\ne0$, then $H(c)=-2c\sum_vf(v)\to-\infty$ for suitable $c\to\pm\infty$.
\end{proof}

\begin{remark}
On infinite graphs, minimizing $H$ over functions of finite energy selects the \emph{finite-energy} solution, which in general differs from the bounded solutions of Section~\ref{sec:existence}; see \cite{Soardi,KLW,Tang} for this theory.
\end{remark}

\section{Infinite graphs: existence}\label{sec:existence}

A graph is \emph{locally finite} if every vertex has finite degree. A connected locally finite graph is countable, and balls $\{w:\dist(v,w)\le k\}$ are finite. The \emph{width} of $G$ relative to $B$ is
\[
\Wid(G,B)=\sup_{v\in V}\dist(v,B),
\]
where $\dist(v,B)$ is the length of a shortest path from $v$ to $B$.

\subsection{Existence without degree bounds}

\begin{lemma}\label{lem:apriori}
Let $G$ be finite and connected, $\emptyset\ne B\subsetneq V$, $W=\Wid(G,B)$, and let $\varphi$ solve \eqref{eq:dirichlet}. Put
\begin{equation}\label{eq:Lambda}
\Lambda=\sum_{u\notin B}|f(u)|+\Bigl(\sum_{u\in B}d_u\varphi_B(u)^2\Bigr)^{1/2}.
\end{equation}
Then $\max_V|\varphi|\le6\Lambda W^2$.
\end{lemma}

\begin{proof}
Note that $|\varphi_B(w)|\le(d_w\varphi_B(w)^2)^{1/2}\le\Lambda$ for $w\in B$, since $d_w\ge1$; as $W\ge1$, the bound holds on $B$. Let $C=\max_{V\setminus B}|\varphi|$. For $u\notin B$, $\sum_{w\sim u}\varphi(w)=d_u\varphi(u)-f(u)$, so
\begin{align*}
\sum_{u\notin B}\sum_{w\sim u}(\varphi(u)-\varphi(w))^2
&=\sum_{u\notin B}\Bigl(-d_u\varphi(u)^2+2\varphi(u)f(u)\Bigr)+\sum_{u\notin B}\sum_{w\sim u}\varphi(w)^2.
\end{align*}
The last sum equals $\sum_{w\in V}n_w\varphi(w)^2$ with $n_w=|\{u\notin B:u\sim w\}|\le d_w$, so it is at most $\sum_{w\notin B}d_w\varphi(w)^2+\sum_{w\in B}d_w\varphi_B(w)^2$. Hence the left side is at most $2C\Lambda+\Lambda^2$, and for every edge $\{u,w\}$ with $u\notin B$,
\begin{equation}\label{eq:edge}
|\varphi(u)-\varphi(w)|\le\sqrt{2C\Lambda}+\Lambda.
\end{equation}
Let $|\varphi(v)|=C$ with $v\notin B$, and let $v=x_0,x_1,\dots,x_\ell=w$ be a shortest path from $v$ to $B$, so $\ell\le W$ and $x_0,\dots,x_{\ell-1}\notin B$. Applying \eqref{eq:edge} along it and using $W\sqrt{2C\Lambda}=\sqrt{C\cdot2W^2\Lambda}\le\frac C2+W^2\Lambda$,
\[
C\le\Lambda+W\sqrt{2C\Lambda}+W\Lambda\le\frac C2+\Lambda+W^2\Lambda+W\Lambda\le\frac C2+3W^2\Lambda,
\]
so $C\le6W^2\Lambda$.
\end{proof}

\begin{theorem}\label{thm:existence}
Let $G$ be locally finite and connected, $B\subsetneq V$ with $W=\Wid(G,B)<\infty$, and suppose
\[
\Lambda=\sum_{w\notin B}|f(w)|+\Bigl(\sum_{w\in B}d_w\varphi_B(w)^2\Bigr)^{1/2}<\infty.
\]
Then \eqref{eq:dirichlet} has a solution $\varphi$ with $|\varphi(w)|\le24\Lambda W^2$ for all $w\in V$.
\end{theorem}

\begin{proof}
Since $W<\infty$, $B\neq\emptyset$; fix $v\in B$. Let $V_k=\{u:\dist(v,u)\le k\}$, $G_k=G[V_k]$ and $B_k=B\cap V_k\ni v$. Each $G_k$ is finite and connected, since geodesics from $v$ stay in $V_k$.

\emph{Claim: $\Wid(G_k,B_k)\le2W$, distances taken in $G_k$.} Let $u\in V_k$ and $\ell=\dist(v,u)\le k$, with geodesic $v=u_0,\dots,u_\ell=u$. If $\ell\le W$, the geodesic itself shows $\dist_{G_k}(u,B_k)\le W$. If $\ell>W$, put $m=\ell-W\ge1$ and choose $w\in B$ and a geodesic $\gamma$ from $u_m$ to $w$ of length at most $W$. Every vertex $x$ of $\gamma$ satisfies $\dist(v,x)\le m+W=\ell\le k$, so $\gamma\subset V_k$ and $w\in B_k$. Following the geodesic from $u$ back to $u_m$ ($W$ steps) and then $\gamma$ gives a path in $G_k$ of length at most $2W$.

Consider the Dirichlet problem on $G_k$ with boundary $B_k$, data $\varphi_B|_{B_k}$ and $f|_{V_k\setminus B_k}$, using the Laplacian of $G_k$. By Proposition~\ref{prop:finite} it has a unique solution; extend it by $0$ to get $\varphi_k\in\F$. If $V_k\subseteq B$ then $\varphi_k=\varphi_B$ on $V_k$ and $|\varphi_k|\le\Lambda$. Otherwise Lemma~\ref{lem:apriori} applies; since degrees in $G_k$ do not exceed those in $G$, the quantity \eqref{eq:Lambda} for $G_k$ is at most $\Lambda$, and $|\varphi_k|\le6\Lambda(2W)^2=24\Lambda W^2$.

By a diagonal argument over the countable set $V$, a subsequence $\varphi_{k_i}$ converges pointwise to some $\varphi$ with $|\varphi|\le24\Lambda W^2$. If $u\notin B$ and $k>\dist(v,u)+1$, all neighbours of $u$ lie in $V_k$ and $u$ has the same degree in $G_k$ as in $G$, so $\Delta\varphi_k(u)=f(u)$; this is a finite linear relation and passes to the limit. On $B$, $\varphi_k=\varphi_B$ eventually.
\end{proof}

For harmonic functions the width assumption is unnecessary.

\begin{corollary}\label{cor:harmonic}
Let $G$ be locally finite and connected, $\emptyset\ne B\subseteq V$, and $C_1\le\varphi_B\le C_2$ on $B$. Then there is $\varphi$, harmonic on $V\setminus B$, with $\varphi|_B=\varphi_B$ and $C_1\le\varphi\le C_2$.
\end{corollary}

\begin{proof}
Repeat the proof of Theorem~\ref{thm:existence} with $f=0$, replacing Lemma~\ref{lem:apriori} by the maximum and minimum principles (Proposition~\ref{prop:max}) on each $G_k$, which give $C_1\le\varphi_k\le C_2$ on $V_k$.
\end{proof}

\subsection{Bounded data on graphs of bounded degree}

\begin{lemma}\label{lem:martingale}
Let $G$ be locally finite and connected, $B\ne\emptyset$, and let $h$ be bounded, harmonic on $V\setminus B$, with $h|_B=0$. Then $|h(v)|\le\bigl(\sup_V|h|\bigr)\,\P_v(\tau=\infty)$ for every $v$. Conversely, $h_\infty(v)=\P_v(\tau=\infty)$ is bounded, harmonic on $V\setminus B$, and vanishes on $B$.
\end{lemma}

\begin{proof}
Harmonicity of $h$ at $x\notin B$ means $h(x)=\E[h(X_{t+1})\mid X_t=x]$, so $h(X_{t\wedge\tau})$ is a bounded martingale. Since $h(X_\tau)=0$ on $\{\tau\le t\}$, $h(v)=\E_v[h(X_{t\wedge\tau})]=\E_v[h(X_t);\tau>t]$, whence $|h(v)|\le\sup|h|\cdot\P_v(\tau>t)$; let $t\to\infty$. The converse follows from the Markov property at time $1$ for $v\notin B$, and $\tau=0$ for $v\in B$.
\end{proof}

\begin{theorem}\label{thm:rw}
Let $G$ be locally finite and connected and $B\neq\emptyset$.
\begin{enumerate}[label=(\roman*)]
\item The homogeneous problem ($f=0$, $\varphi_B=0$) has a nonzero bounded solution if and only if $\P_v(\tau=\infty)>0$ for some $v$. Hence two bounded solutions of \eqref{eq:dirichlet} always coincide if and only if $\P_v(\tau<\infty)=1$ for all $v$.
\item Suppose $d_v\le D$ for all $v$, $W=\Wid(G,B)<\infty$, and $f$, $\varphi_B$ are bounded. Then \eqref{eq:rwformula} defines the unique bounded solution of \eqref{eq:dirichlet}, and
\[
\sup_V|\varphi|\le\sup_B|\varphi_B|+W D^{W}\sup_{V\setminus B}|f|.
\]
\end{enumerate}
\end{theorem}

\begin{proof}
(i) is Lemma~\ref{lem:martingale}: the difference of two bounded solutions is a bounded solution of the homogeneous problem.

(ii) From any vertex $x$, the walk follows a fixed shortest path to $B$, of length at most $W$, with probability at least $D^{-W}$, so $\P_x(\tau\le W)\ge D^{-W}$. By the Markov property $\P_v(\tau>jW)\le(1-D^{-W})^j$, hence $\tau<\infty$ almost surely and
\[
\E_v\tau=\sum_{t\ge0}\P_v(\tau>t)\le W\sum_{j\ge0}\P_v(\tau>jW)\le WD^W.
\]
The sum in \eqref{eq:rwformula} is bounded by $\tau\sup|f|$ since $d\ge1$, so both expectations converge absolutely and give the stated bound. For $v\in B$, $\tau=0$ and $\varphi(v)=\varphi_B(v)$. For $v\notin B$, $\tau\ge1$, and conditioning on $X_1$ gives $\varphi(v)=f(v)/d_v+d_v^{-1}\sum_{w\sim v}\varphi(w)$, which is $\Delta\varphi(v)=f(v)$. Uniqueness among bounded solutions follows from (i).
\end{proof}

\begin{remark}
Theorems~\ref{thm:existence} and~\ref{thm:rw} are complementary. The first needs no degree bound and gives a polynomial bound in $W$, but requires data decaying at infinity. The second requires bounded degree and gives an exponential constant, but allows bounded data and yields uniqueness among bounded functions. When both apply, the solution of Theorem~\ref{thm:existence} is bounded and therefore coincides with that of Theorem~\ref{thm:rw}.
\end{remark}

\section{Ends, parabolicity and uniqueness}\label{sec:ends}

\begin{definition}
An \emph{end} of $G$ relative to $B$ is an infinite connected component of $G[V\setminus B]$. An end $E$ is \emph{parabolic} if $\P_v(\tau<\infty)=1$ for $v\in E$, and \emph{non-parabolic} otherwise. Set
\[
\Hz=\{z\in\F:\Delta z=0\text{ on }V\setminus B,\ z|_B=0\}.
\]
\end{definition}

Distinct components of $G[V\setminus B]$ are not joined by edges, so the walk started in an end $E$ stays in $E$ until time $\tau$; parabolicity of $E$ depends only on $E$ and its edges to $B$. If $\varphi$ solves \eqref{eq:dirichlet}, the solution set is $\varphi+\Hz$.

\begin{remark}\label{rem:recurrence}
If $\partial E$ is finite, then $E$ is parabolic if and only if the graph $G[E\cup\partial E]$ is recurrent: a recurrent walk hits every finite set almost surely, while a transient walk started at a suitable vertex avoids a given finite set with positive probability \cite{Woess,LyonsPeres}. If $\partial E$ is infinite, $E$ may be parabolic even when $G[E]$ is transient. Parabolicity of $E$ is also equivalent to infinite effective resistance between $\partial E$ and infinity in $E\cup\partial E$ \cite{DoyleSnell,Soardi}.
\end{remark}

\begin{figure}[ht]
\centering
\begin{tikzpicture}[scale=1.0,
  vtx/.style={circle,fill=black,inner sep=1.6pt},
  bdy/.style={rectangle,fill=black,inner sep=2.6pt}]
  \node[bdy] (b1) at (0,0) {};
  \node[bdy] (b2) at (0,1.4) {};
  \foreach \i in {1,2,3,4} {
    \node[vtx] (l\i) at (-\i,0.7) {};
    \node[vtx] (r\i) at (\i,0.7) {};
  }
  \node[vtx] (c1) at (-0.55,2.4) {};
  \node[vtx] (c2) at (0.55,2.4) {};
  \draw (b1)--(b2);
  \draw (b1)--(l1)--(b2); \draw (b1)--(r1)--(b2);
  \draw (l1)--(l2)--(l3)--(l4); \draw (r1)--(r2)--(r3)--(r4);
  \draw[dashed] (l4)--(-5,0.7); \draw[dashed] (r4)--(5,0.7);
  \draw (b2)--(c1)--(c2)--(b2);
  \node at (-2.5,0.25) {$E_1$};
  \node at (2.5,0.25) {$E_2$};
  \node at (0,2.85) {$C$};
  \node at (0.45,-0.05) {$B$};
\end{tikzpicture}
\caption{The boundary $B$ (squares) splits $G$ into a finite component $C$ and two ends $E_1,E_2$. Theorem~\ref{thm:ends} gives $z_1\in\Hz$ positive on $E_1$ and zero elsewhere, and $z_2$ likewise for $E_2$. Both ends here are rays, hence parabolic, so by Theorem~\ref{thm:parabolic} $z_1$ and $z_2$ are unbounded and the bounded solution is unique.}
\label{fig:ends}
\end{figure}
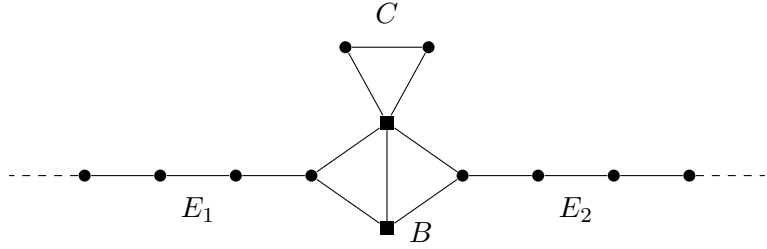

\begin{theorem}\label{thm:noends}
Let $G$ be locally finite and connected, and $B\ne\emptyset$. If $G$ has no ends relative to $B$, then for all $f$ and $\varphi_B$ the Dirichlet problem \eqref{eq:dirichlet} has exactly one solution.
\end{theorem}

\begin{proof}
Every component $C$ of $G[V\setminus B]$ is finite, and $\partial C\ne\emptyset$ because $G$ is connected and $B\ne\emptyset$. Every neighbour of a vertex of $C$ lies in $C\cup\partial C$, so the equations of \eqref{eq:dirichlet} at vertices of $C$ involve only values on $C\cup\partial C$, and they coincide with the Dirichlet problem on the finite connected graph $G[C\cup\partial C]$ with boundary $\partial C$ (degrees of vertices of $C$ are unchanged). By Proposition~\ref{prop:finite}(i) this problem has exactly one solution for each $C$, and the problems for different $C$ are independent.
\end{proof}

If $B=\emptyset$ and $G$ is infinite, $G$ itself is an end; if $B=\emptyset$ and $G$ is finite, Proposition~\ref{prop:finite} applies.

\begin{theorem}\label{thm:ends}
Let $G$ be locally finite and connected with $l\in[0,\infty]$ ends relative to $B$. For each end $E$ there is $z_E\in\Hz$ with $z_E>0$ on $E$ and $z_E=0$ off $E$. In particular $\dim\Hz\ge l$, and if \eqref{eq:dirichlet} has a solution, its solution set has dimension at least $l$.
\end{theorem}

\begin{proof}
If $B=\emptyset$ and $G$ is infinite, then $l=1$ and the constant $1$ works. Assume $B\ne\emptyset$ and let $E$ be an end. Since the $z_E$ for distinct ends have disjoint supports, they are linearly independent, which gives the dimension bound (for $l=\infty$, $\Hz$ is infinite-dimensional).

\emph{Reduction.} Let $\Gamma=G[E\cup\partial E]$. Vertices of $E$ have the same neighbours, hence the same degree, in $\Gamma$ as in $G$. Suppose $z$ on $V(\Gamma)$ is harmonic on $E$, vanishes on $\partial E$ and is positive on $E$; extend it by zero. At a vertex of $E$ it is harmonic in $G$. At a vertex $x$ of $V\setminus B$ outside $E$, all neighbours of $x$ lie in its own component or in $B$, where the extension is zero, so $\Delta z(x)=0$. Hence the extension lies in $\Hz$. It remains to construct $z$ on $\Gamma$.

\emph{Construction.} Fix $v\in\partial E$ and $w_1\in E$ with $w_1\sim v$, and let $\rho(x)=\dist_\Gamma(v,x)$. For $k\ge2$ let $\Gamma_k$ be the subgraph of $\Gamma$ induced by $\{\rho\le k\}$, which is finite and connected, and set
\[
I_k=\{x\in E:\rho(x)\le k-1\},\qquad S_k=\{x\in E:\rho(x)=k\},\qquad B_k=(\partial E\cap\{\rho\le k\})\cup S_k.
\]
Neighbours of $x\in I_k$ satisfy $\rho\le k$, so they lie in $\Gamma_k$, and $x$ has the same degree there.

\emph{A path from $w_1$ to $S_k$ inside $I_k$.} Since $E$ is infinite and $\{\rho\le k\}$ is finite, there is $x\in E$ with $\rho(x)>k$, and a path from $w_1$ to $x$ inside $E$ (which is connected). Along it $\rho$ changes by at most $1$ per step and starts at $\rho(w_1)=1$, so the first vertex with $\rho=k$ lies in $S_k$ and all earlier vertices lie in $I_k$. (This replaces the ray argument of the earlier version, which mixed distances in $\Gamma$ and in $G$.)

Let $\varphi_k$ be the solution on $\Gamma_k$, with boundary $B_k$, of
\[
\Delta\varphi_k=0\text{ on }I_k,\qquad \varphi_k=0\text{ on }\partial E\cap\{\rho\le k\},\qquad \varphi_k=1\text{ on }S_k,
\]
which exists and is unique by Proposition~\ref{prop:finite}(i). By Proposition~\ref{prop:max}, $0\le\varphi_k\le1$. If $\varphi_k(w_1)=0$, then $w_1\in I_k$ is an interior minimum, so all its neighbours vanish; propagating along the path above, we reach a vertex of $S_k$ with value $0$, a contradiction. So $\varphi_k(w_1)>0$, and
\[
z_k=\varphi_k/\varphi_k(w_1)\ \text{on }\{\rho\le k\},\qquad z_k=0\ \text{elsewhere on }V(\Gamma),
\]
satisfies $z_k\ge0$, $z_k(w_1)=1$, $z_k=0$ on $\partial E$, and $\Delta z_k=0$ on $I_k$.

\emph{Harnack-type bound.} Fix $x\in E$ and a path $w_1=x_0,x_1,\dots,x_m=x$ in $E$. For $k\ge m+2$ all $x_i$ with $i<m$ satisfy $\rho(x_i)\le i+1\le k-1$, so lie in $I_k$. Harmonicity at $x_i$ and $z_k\ge0$ give $z_k(x_{i+1})\le\sum_{u\sim x_i}z_k(u)=d_{x_i}z_k(x_i)$, hence
\[
z_k(x)\le\prod_{i=0}^{m-1}d_{x_i}\qquad(k\ge m+2).
\]

\emph{Limit.} By a diagonal argument over the countable set $V(\Gamma)$, a subsequence $z_{k_i}$ converges pointwise to some $z\ge0$ with $z(w_1)=1$ and $z|_{\partial E}=0$. Each $x\in E$ lies in $I_k$ for large $k$, and $\Delta z_k(x)=0$ is a finite linear relation, so $\Delta z(x)=0$. Finally, if $z(x)=0$ for some $x\in E$, the minimum principle propagates $z=0$ through the connected set $E$ to $w_1$, a contradiction; so $z>0$ on $E$.
\end{proof}

\begin{theorem}\label{thm:parabolic}
Let $G$ be locally finite and connected, $B\ne\emptyset$, and let $E$ be an end relative to $B$.
\begin{enumerate}[label=(\roman*)]
\item If $E$ is non-parabolic, the function $h_E(x)=\P_x(\tau=\infty)$ for $x\in E$, $h_E=0$ off $E$, is a bounded element of $\Hz$, positive on $E$.
\item If $E$ is parabolic, every bounded element of $\Hz$ vanishes on $E$. In particular every $z_E$ as in Theorem~\ref{thm:ends} is unbounded on $E$: there are $x_n\in E$, eventually leaving every finite set, with $z_E(x_n)\to\infty$.
\item Bounded solutions of \eqref{eq:dirichlet} are unique for all data if and only if every end is parabolic.
\end{enumerate}
\end{theorem}

\begin{proof}
(i) By Lemma~\ref{lem:martingale}, $h_\infty(x)=\P_x(\tau=\infty)$ is harmonic on $V\setminus B$ and zero on $B$. The function $h_E$ agrees with $h_\infty$ on $E\cup B$ and is zero elsewhere; it is harmonic at vertices of $E$ since their neighbours lie in $E\cup B$, and at other vertices of $V\setminus B$ by the argument in the reduction step of Theorem~\ref{thm:ends}. It is nonzero because $E$ is non-parabolic, and positive on $E$ by the minimum principle on the connected set $E$.

(ii) By Lemma~\ref{lem:martingale}, a bounded $h\in\Hz$ satisfies $|h(x)|\le\sup|h|\cdot\P_x(\tau=\infty)=0$ for $x\in E$. Since $z_E>0$ on $E$, it cannot be bounded; as $z_E$ is finite on every finite set, a sequence with $z_E(x_n)\to\infty$ eventually leaves every finite set.

(iii) By Theorem~\ref{thm:rw}(i), uniqueness holds if and only if $\P_x(\tau<\infty)=1$ for all $x$. For $x$ in a finite component $C$ of $G[V\setminus B]$ this always holds: from any vertex of $C$ the walk reaches $\partial C$ within $|C|$ steps with probability at least some $\varepsilon>0$, so $\P_x(\tau>j|C|)\le(1-\varepsilon)^j$. For $x$ in an end, it holds exactly when the end is parabolic.
\end{proof}

Theorem~\ref{thm:parabolic}(ii) is the discrete counterpart of the Li--Tam construction on parabolic ends \cite{LiTam}, and (i) of the bounded harmonic functions on non-parabolic ends.

\begin{example}\label{ex:table}
The lower bound $\dim\Hz\ge l$ can be attained or far from sharp.
\begin{enumerate}[label=(\alph*)]
\item $G=\Z$, $B=\{0\}$. Two ends, both parabolic. $\Hz=\{\varphi: \varphi(i)=ai\ (i>0),\ \varphi(i)=bi\ (i<0)\}$ has dimension $2=l$, and contains no nonzero bounded function.
\item The ladder $\Z\times\{0,1\}$ with $B=\Z\times\{0\}$. One end, the top row, which is parabolic (from the top row the walk steps down with probability $1/3$ at each step). On the top row $3\varphi(n)=\varphi(n-1)+\varphi(n+1)$, so $\Hz=\{ar^n+br^{-n}\}$ with $r=(3+\sqrt5)/2$, of dimension $2>l=1$. Here $W=1$ and $D=3$, so Theorem~\ref{thm:rw} applies, consistent with the absence of bounded elements.
\item $G=\Z^2$, $B=\{0\}$. One parabolic end, yet $\dim\Hz=\infty$. Indeed, the equation $\Delta\varphi=0$ at $(x,y)$ can be solved for $\varphi(x+1,y)$, so a function harmonic on all of $\Z^2$ may be prescribed arbitrarily on the two columns $x=0$ and $x=1$, and is then uniquely determined. Imposing $\varphi(0,0)=0$ leaves an infinite-dimensional space inside $\Hz$ (for example $x$, $y$, $xy$, $x^2-y^2$).
\item $G=\Z^3$, $B=\{0\}$. One non-parabolic end; $\dim\Hz=\infty$ by the argument of (c) applied to two parallel planes, and $\Hz$ contains the bounded function $\P_v(\tau=\infty)=1-\mathcal G(v,0)/\mathcal G(0,0)$, where $\mathcal G$ is the Green function.
\item The rooted binary tree (root of degree $2$, all other vertices of degree $3$), $B=\{\text{root}\}$. Two non-parabolic ends. At a vertex $u$ with parent $p$ and children $c,c'$, harmonicity reads $3\varphi(u)=\varphi(p)+\varphi(c)+\varphi(c')$; choosing $\varphi(c)$ freely at each vertex and solving for $\varphi(c')$ shows $\dim\Hz=\infty$.
\end{enumerate}
\end{example}

Examples (b) and (c) show that $\dim\Hz$ is governed by the geometry of each end at infinity rather than by the number of ends. The natural invariants are the Martin and Poisson boundaries \cite{Woess}; prescribing data at the ends, as in \cite{Perkins}, restores uniqueness in the non-parabolic case.

\section{Open questions}\label{sec:questions}

\begin{question}
Theorem~\ref{thm:rw}(ii) uses a degree bound. If $\Wid(G,B)<\infty$ but degrees are unbounded, must every end be parabolic, so that bounded solutions are unique?
\end{question}

\begin{question}
For which graphs with finitely many ends is $\dim\Hz$ finite? Examples~\ref{ex:table}(a)--(b) give finite dimension and (c)--(e) infinite dimension. Is finiteness equivalent to each end having linear volume growth?
\end{question}

\begin{question}
Under growth restrictions, spaces of harmonic functions become finite-dimensional on manifolds and groups of polynomial growth \cite{ColdingMinicozzi,Kleiner}. What is the dimension of $\{z\in\Hz:|z(x)|=O(\dist(x,B)^k)\}$ in terms of the ends of $G$?
\end{question}

\begin{question}
Describe the finite-energy solutions of \eqref{eq:dirichlet} on infinite graphs, and compare them with the bounded solutions of Theorem~\ref{thm:rw} and the solutions of Theorem~\ref{thm:existence}.
\end{question}

\section*{Acknowledgement}
This revised version is prepared by an AI assistant (Claude).

\end{document}